\documentclass[12pt]{article}
\usepackage{amsthm}
\usepackage{latexsym}
\usepackage[latin1]{inputenc}
\usepackage{amssymb}
\usepackage{amsfonts}

\newcommand{\bR}{{\mathbb{R}}}
\newcommand{\bN}{{\mathbb{N}}}

\usepackage[T1]{fontenc}           
\usepackage{textcomp}              
\usepackage{epsfig}        
\usepackage{tikz}

\newtheorem{theorem}{Theorem}[section]
\newtheorem{lemma}{Lemma}[section]
\newtheorem{corollary}{Corollary}[section]

\theoremstyle{definition} 
\newtheorem{remark}{Remark}[section]
\begin{document}


\title{Illuminating and covering convex bodies}

\author{Horst Martini\\
\small Technische Universit\"at Chemnitz\\
\small Fakult\"at f\"ur Mathematik\\
\small D-09107 Chemnitz, Germany\\
{\small\tt horst.martini@mathematik.tu-chemnitz.de}\\[2ex]
Christian Richter\\
\small Friedrich-Schiller-Universit\"at Jena\\
\small Mathematisches Institut\\
\small D-07737 Jena, Germany\\
{\small\tt christian.richter@uni-jena.de}\\[2ex]
Margarita Spirova\\
\small Technische Universit\"at Chemnitz\\
\small Fakult\"at f\"ur Mathematik\\
\small D-09107 Chemnitz, Germany\\
{\small\tt margarita.spirova@mathematik.tu-chemnitz.de}}

\date{\today}

\maketitle

\begin{abstract}
Covering numbers of convex bodies based on homothetical copies and related illumination numbers are well-known in combinatorial geometry and,
for example, related to Hadwiger's famous covering problem. Similar numbers can be defined by using proper translates instead of
homothets, and even more
related concepts make sense. On these lines we introduce some new covering and illumination numbers of convex bodies, present
their properties and compare them with each other as well as  with already known numbers. Finally, some suggestive examples illustrate that these new
illumination numbers are interesting and non-trivial.
\end{abstract}




\section{Introduction}

Our considerations refer to the combinatorial geometry of convex bodies.
The famous covering problem of Hadwiger asks for the minimum number of
smaller homothetical copies of a convex body $K$ in $\mathbb{R}^n$ necessary to
cover $K$. Due to ideas of Boltyanski and Hadwiger (see the surveys
\cite{Be-1}, § 34 in \cite{B-M-S}, \cite{Ma-So}, \cite{Sz}, \cite{Be-2}, and Chapters 3 and 9 of \cite{Be-3}), this problem can be equivalently formulated in terms of parallel or suitable central illumination of $K$, where the minimal number of  directions or light sources illuminating the whole of  $K$  is equal to the minimal number of homothets covering it. There are further  related (but not equivalent)  covering problems referring to convex bodies, such as covering a convex body $K$ by the minimal number of proper translates of it (see, e.g., \cite{L-M-S}), and many natural modifications; see again the surveys above and, furthermore, \cite{So}, \cite{Sz-Ta}, \cite{We}, \cite{Bo-Ma-So},   \cite{De}, \cite{M-W-1}, \cite{Ma-We},   \cite{Swa}, \cite{Be-B-Ki},  \cite{L-So},  \cite{Be-Li},  \cite{Mo-Se},  \cite{Nasz}, and   \cite{Ki-We}. It turns out that, also in this wider sense, again the covering problems can be  expressed in terms of correspondingly equivalent illumination and visibility notions. Inspired by this general observation, we introduce some new interesting covering and illumination problems, which are also pairwise corresponding  to each other. We will collect basic properties of them, and we will also compare these problems with the already known ones. Finally, a collection of examples will show how interesting and non-trivial   these new problems are, such that they certainly will create new research activities.


\section{Notions and definitions}

Let $K \subseteq \bR^n$ be an $n$-dimensional \emph{convex body}, i.e., a compact, convex set with nonempty interior  in Euclidean space $\bR^n$. By $o$ we denote the origin. As usual, we use the abbreviations $\mathrm{conv}$, $\mathrm{bd}$, $\mathrm{cl}$, $\mathrm{int}$, $\mathrm{vert}$, and $\mathrm{dist}$ for convex hull, boundary, closure, interior, vertex set, and distance, respectively.  The following covering and illumination numbers are well-known (see, e.g., \cite{Be-1}, Chapter VI of \cite{B-M-S}, \cite{Sz}, \cite{Ma-So}, \cite{Be-2},  and Chapters 3 and 9 of \cite{Be-3}). Let
\[
\begin{array}{r@{\;}l}
b(K): = \min \{m :& K \mbox{ can be covered by } m \mbox{ smaller}\\
&\mbox{homothetical copies of itself }\}\,,
\end{array}
\]
and let
\[
b'(K) : = \min\left\{m : \exists t_1, \dots, t_m \in \bR^n \left(K \subseteq \bigcup\limits^m_{i=1} {\rm int} (K) + t_i \right)\right\}\,.
\]
We say that the direction $l \in \bR^n \setminus \{o\}$ \emph{illuminates} $x \in K$ if
there exists $\varepsilon > 0$ such that $x + \varepsilon l \in {\rm int}(K)$, and we
introduce the corresponding illumination number
\[
c(K) : = \min\{m : \exists l_1, \dots,l_m \in \bR^n \setminus \{o\} \,\, \forall x \in K \,\,\exists i\,\, (l_i \, \mbox{ illuminates } \, x)\}\,.
\]
Further on, we say that $y \in \bR^n \setminus K$ \emph{c-illuminates} $x \in K$ if $x-y$ illuminates $x$ (i.e., if $x+\varepsilon (x-y)\in {\rm int}(K)$ for some $\varepsilon> 0$, called \emph{central illumination}), and we introduce
\[
c'(K) : = \min \{m: \exists y_1,\dots,y_m \in \bR^n
\setminus K \,\, \forall x \in K \,\,\exists i\,\,(y_i\; c\mbox{-illuminates } x)\}\,.
\]
Theorem 34.3 in \cite{B-M-S} says that for any convex body $K \subseteq \bR^n$ we have
\begin{equation}\label{1}
b(K) = b'(K) = c(K) = c'(K)\,.
\end{equation}
We continue with some notions and results from \cite{L-M-S}. Again $K \subseteq \bR^n$ be a convex body. Let
\[
t(K) : = \min \left\{m : \exists t_1,\dots,t_m \in \bR^n \setminus \{o\} \, \left(K \subseteq \bigcup\limits^m_{i=1} K + t_i \right)\right\}\,.
\]
Thus $t(K)$ is the usual \emph{$t$-covering number} of $K$ (\emph{translative covering}),
considered in \cite{L-M-S}. Note that
\[
t(K) \le b'(K)\,,
\]
and this inequality may be strict, as we can see by the cube $[-1,1]^n \subseteq \bR^n$,
where $t([-1,1]^n)=2$ and $b'([-1,1]^n)=2^n$.

We say that $l \in \bR^n \setminus \{o\}$ \emph{$t$-illuminates} $x \in K$ if $x + \varepsilon l \in K$ for some $\varepsilon > 0$, and we introduce
\[
i(K) : = \min \{m: \exists l_1,\dots,l_m \in \bR^n \setminus \{o\} \,\,
\forall x \in K \,\,\exists i\,\, (l_i\; t\mbox{-illuminates } x)\}\,.
\]
Theorem 3.1 in \cite{L-M-S} says that
\begin{equation}\label{2}
i(K) \le t(K)\,.
\end{equation}
More precisely, $K \subseteq \bigcup \limits^m_{i=1} K + t_i, t_i \not= o$, implies that the system $\{-t_1,\dots,-t_m\}$ $t$-illuminates $K$. Further on, it should be noticed that there are convex bodies $K \subseteq \bR^n,
n \ge 3$, strictly satisfying $i(K) < t(K)$, for example compact double cones over $(n-1)$-balls like $K = {\rm conv}(\{\pm(0,\dots,0,1)\} \cup \{(\xi_1,\dots,\xi_{n-1},0) : \xi^2_1 + \dots + \xi^2_{n-1} \le 1\})$.

Before studying more covering and illumination quantities, we fix
some notation. We write $\|\cdot\|$  for the Euclidean norm in $\mathbb{R}^n$.
The \emph{closed and open ball} of radius $r$ with center $x_0 \in \mathbb{R}^n$ are
$B(x_0,r)=\{x \in \mathbb{R}^n: \|x-x_0\| \le r\}$ and $B^o(x_0,r)=\{x \in
\mathbb{R}^n: \|x-x_0\| < r\}$, respectively, and  $S^{n-1}={\rm bd}\,(B(o,1))= \{x \in \mathbb{R}^n: \|x\|=1\}$ is
%
%
the \emph{unit sphere} of $\mathbb{R}^n$. For every convex body $K \subseteq \mathbb{R}^n$, there
exists a uniquely determined closed ball of minimal radius that contains $K$.
This \emph{circumball} of $K$ is denoted by $B(c_K,R(K))$, i.e., $c_K$ and $R(K)$
are the \emph{circumcenter} and the \emph{circumradius} of $K$, respectively. Note that
\begin{equation}
\label{eq_cK}
c_K \in K\,.
\end{equation}
Indeed, if we assume $c_K \notin K$, then $c_K$ were strictly separated from $K$ by a hyperplane $H$: $K \subseteq {\rm int} (H^-)$, $c_K \in {\rm int} (H^+)$, where $H^-, H^+$ are the closed half-spaces generated by $H$. Then
$K \subseteq B(c_K, R (K)) \cap H^-$. But the circumradius of $B (c_K, R(K)) \cap H^-$ is strictly smaller than  $R(K)$, a contradiction.


\section{The notion of $\mathbf{t}$-central illumination}

We say that $y \in \bR^n \setminus K$ $t$-$c$-\emph{illuminates} $x \in K$ if $x-y$ $t$-illuminates
$x$, i.e., if $x + \varepsilon (x-y) \in K$ for some $\varepsilon > 0$. Due to this, we introduce
\[
c''(K) : = \min \{m: \exists y_1,\dots,y_m \in \bR^n \setminus K \,\, \forall x \in K \,\,\exists i\,\,(y_i \, t\mbox{-}c\mbox{-illuminates } x)\}\,.
\]

\begin{theorem}\label{theo1}
For any convex body $K \subseteq \bR^n$ we have $c''(K) = c'(K)$.
\end{theorem}

\begin{proof}
It is clear that  $c''(K) \le c'(K)$. Thus we prove now $c'(K) \le c''(K)$. Without loss of generality, we have $n \ge 2$ (for $n=1, c'(K) = c''(K) = 2$), and also
$o \in {\rm int}\,(K)$. We want to show that if $x \in K$ is $t$-$c$-illuminated by $y \in \bR^n \setminus K$, then $x$ is also $c$-illuminated by $\lambda y$, for $\lambda > 1$.
Since $x + \varepsilon (x-y) \in K$ and $o \in {\rm int}\,(K)$, we have
 \[
 x' = \frac{\lambda}{\lambda + \varepsilon (\lambda-1)} \left(x+\varepsilon(x-y)\right) \in \mbox{ int}(K)\,;
 \]
 note that $\lambda + \varepsilon (\lambda-1) > \lambda > 0$. With $\overline{\varepsilon} = \frac{\varepsilon}{\lambda + \varepsilon(\lambda-1)} > 0$ we have
 \[
 \begin{array}{lll}
 x + \overline{\varepsilon} (x-\lambda y) & = & \frac{1}{\lambda + \varepsilon (\lambda-1)} \left((\lambda + \varepsilon(\lambda-1)) x + \varepsilon(x-\lambda y)\right)\\[1ex]
 & = & \frac{1}{\lambda + \varepsilon (\lambda-1)} \left(\lambda x + \varepsilon \lambda x - \varepsilon \lambda y \right)\\[1ex]
 & = & \frac{\lambda}{\lambda + \varepsilon(\lambda-1)} \left(x + \varepsilon (x-y)\right) = x' \in \mbox{int}\,(K)\,.
 \end{array}
 \]
 Thus $x \in K$ is $c$-illuminated by $\lambda y$.

 Summarizing, we see that if every $x \in K$ is $t$-$c$-illuminated by one of the points $y_1,\dots,y_{c''(K)} \in \bR^n \setminus K$, then each $x \in K$ is also $c$-illuminated by one of
the points $2y_1,\dots,2y_{c''(K)}$, implying $c'(K) \le c''(K)$.
\end{proof}

 \begin{remark}\label{rem1.1}

    \textbf{(A)} More precisely we have shown: if points $y_1,\dots,y_m \in \bR^n \setminus K$ are sufficient for the $t$-$c$-illumination of $K$, and if $\lambda_1,\dots,\lambda_m > 1$ and $c_1,\dots,c_m \in \,{\rm int}\,(K)$,
   then the points $c_i + \lambda_i(y_i-c_i), \, 1 \le i \le m$, are sufficient for the $c$-illumination of $K$.

  \textbf{(B)} The notion of $t$-$c$-illumination is, in some sense, more elementary than that of $c$-illumination, since the light rays have to pass not necessarily through int$(K)$.
\end{remark}


\section{Finer quantities of covering and illumination}

We say that $y \in \bR^n \setminus \{o\}$ $\hat{t}$-\emph{illuminates} $x \in K$ if
$x + y \in K$, and we call this, in verbal form,
\emph{strict} $t$-\emph{illumination} instead of $t$-illumination.

Clearly, we have
\[
\begin{array}{r@{\,\!}l}
t(& K) = \min \left\{m: \exists t_1,\dots,t_m \in \bR^n \setminus \{o\} \;\; \left(K \subseteq \bigcup \limits^m_{i=1} K+t_i\right)\right\}\\[2ex]
& = \min \left\{m : \exists t_1,\dots,t_m \in \bR^n \setminus \{o\} \,\, \forall x \in K \,\,\exists i \left(x-t_i \in K\right)\right\}\\[1ex]
& = \min \left\{m : \exists y_1,\dots,y_m \in \bR^n \setminus \{o\} \,\, \forall x \in K\,\, \exists i \left( y_i\,\, \hat{t}\mbox{-illuminates } x\right)\right\}\\[1ex]
& =  \min \left\{m: \exists r_1,\dots,r_m \in S^{n-1}\,\, \exists \varepsilon > 0 \,\, \forall x \in K\,\, \exists i \left( \varepsilon r_i \,\,\hat{t}\mbox{-illuminates } x \right)\right\}.
\end{array}
\]

For $r \in S^{n-1}$, $\varepsilon > 0$, and $x \in K$, we say that $r$ $\varepsilon$-$t$-\emph{illuminates} $x$ if $x + \varepsilon r \in K$ (\emph{quantified $t$-illumination}). With this notion we obtain
\[
\begin{array}{rcc@{\;}l}
t(K) & = & \min \{m: &
\exists r_1,\dots,r_m \in S^{n-1} \,\,\exists \varepsilon > 0 \,\, \\[.5ex]
&&& \forall x \in K\,\, \exists i \left(r_i \,\, \varepsilon\mbox{-}t\mbox{-illuminates } \, x\right)\}\\[1ex]
& = & \multicolumn{2}{l}{ \min\limits_{\varepsilon > 0} \, i (K, \varepsilon)\,,}
\end{array}
\]
where
\[
i(K,\varepsilon) :=\left\{ \begin{array}{l}

\min \left\{m : \exists r_1,\dots,r_m \in S^{n-1} \, \forall x \in K \, \exists i \left(r_i\,\, \varepsilon\mbox{-}t\mbox{-illuminates } x \right)\right\}\\[.5ex]
\multicolumn{1}{r}{ \mbox{ if the minimum over a non-empty set is meant},}\\[1ex]
\infty \quad \mbox{ otherwise}\,.
\end{array}\right.
\]

\begin{remark}\label{rem3.1}

\textbf{(A)} The number $i(K,\varepsilon)$ can be interpreted as \emph{quantified t-illumination number} of $K$.

\textbf{(B)} We know that $i(K) \le t(K)$ and that there are examples of convex bodies
$K \subseteq \bR^n, n \ge 3$, satisfying $i(K) < t(K)$. So we have
\[
i(K) \le \min \limits_{\varepsilon > 0} i(K,\varepsilon)\,,
\]
and there are examples $K$ with
\[
i(K) < \min \limits_{\varepsilon > 0} i(K, \varepsilon)\,.
\]
\textbf{(C)} Analogously we quantify the classical illumination number $c(K)$. Let
$r \in S^{n-1}$, $\varepsilon > 0$, $x \in K$. We say that $r$ $\varepsilon$-\emph{illuminates} $x$
if $x + \varepsilon r \in {\rm int} (K)$, and we introduce
\[
c(K,\varepsilon) :=\left\{ \begin{array}{l}
\min \left\{m : \exists r_1,\dots,r_m \in S^{n-1} \, \forall x \in K \, \exists i \left(r_i\,\, \varepsilon\mbox{-illuminates } x \right)\right\}\\[.5ex]
\multicolumn{1}{r}{ \mbox{ if the minimum over a non-empty set is meant},}\\[1ex]
\infty \quad \mbox{ otherwise}\,.
\end{array}\right.
\]

\textbf{(D)} Note that $i(K,\varepsilon)$ and $c(K,\varepsilon)$ are metric
%
%
quantities, whereas the numbers $b(K)$, $b'(K)$, $c(K)$, $c'(K)$, $c''(K)$, $t(K)$, and $i(K)$ are invariant under affine transformations of $K$. Thus, suitably extending related problems to normed spaces might be an interesting task.

\textbf{(E)} One could also investigate quantified versions of the numbers
$c'(K)$ and $c''(K)$ of central illumination.
\end{remark}

\begin{theorem}\label{theo3.1}
The following relations  are satisfied for every convex body $K \subseteq \mathbb{R}^n$.
\begin{enumerate}
\item[(i)] For all $\varepsilon_1, \varepsilon_2$ with $0 < \varepsilon_1 < \varepsilon_2$,
%
%
%
%
\[
\begin{array}{ccccc}
c(K) &\le& c(K,\varepsilon_1) &\le& c(K,\varepsilon_2) \\[.5ex]
 \tikz{\draw (3.3mm,0mm) -- (3.3mm,2.6mm) (0mm,2.6mm) -- (1.2mm,.1mm) -- (2.4mm,2.6mm);}
&&
\tikz{\draw (3.3mm,0mm) -- (3.3mm,2.6mm) (0mm,2.6mm) -- (1.2mm,.1mm) -- (2.4mm,2.6mm);}
&&
\tikz{\draw (3.3mm,0mm) -- (3.3mm,2.6mm) (0mm,2.6mm) -- (1.2mm,.1mm) -- (2.4mm,2.6mm);}
\\
i(K) &\le& i(K,\varepsilon_1) &\le& i(K,\varepsilon_2). \\[.5ex]
\end{array}
\]
\item[(ii)] $c(K) = \min\limits_{\varepsilon > 0}c(K,\varepsilon)$.
\item[(iii)] The circumradius $R(K)$ can be expressed as
\[
\sup\{\varepsilon > 0 : i (K,\varepsilon) < \infty\} = \sup\{\varepsilon > 0: c(K,\varepsilon) < \infty\} = R(K)\,.
\]
\end{enumerate}
\end{theorem}

\begin{proof}
(i) Suppose that the vector $r \in S^{n-1}$ $\varepsilon_2$-illuminates
$x \in K$, i.e., $x + \varepsilon_2 r \in  \mathrm{int} (K)$. Then
$x + \varepsilon_1 r \in \mathrm{int} (K)$, because $\varepsilon_1 < \varepsilon_2$
and $K$ is convex, and $r$  $\varepsilon_1$-illuminates $x$. This proves
the monotonicity of $c(K,\cdot)$, and the monotonicity of $i(K,\cdot)$ is obtained analogously. The other inequalities in (i)
are obvious.

(ii) By (\ref{1}),
\[
c(K)=b'(K)=\min \left\{m : \exists t_1,\dots,t_m \in \bR^n \left(K \subseteq \bigcup \limits^m_{i=1}\, {\rm int } (K) + t_i\right)\right\}\,.
\]

\emph{Note:} We can assume that the vectors $t_1,\dots,t_m$ are chosen from $\bR^n \setminus \{o\}$.

For proving that, we suppose $t_m=o$. The open set int$(K) + t_m = {\rm int}(K)$ is needed for covering the compact set
\[
C: = K \setminus \bigcup \limits^{m-1}_{i=1} {\rm int} (K) + t_i\,.
\]
Without loss of generality, we have $C \not= \emptyset$. The continuous function $f: C \to \bR$, given by
\[
f(x) = {\rm dist} \, (x,\bR^n \setminus {\rm int} (K)) = \inf \{\|x-y\| : y \in \bR^n \setminus {\rm int}(K)\}\,,
\]
is everywhere positive and attains its minimum on $C$. Therefore there exists $\delta > 0$ such that
\[
\forall x \in C \,\,({\rm dist} (x, \bR^n \setminus {\rm int}(K)) > \delta)\,.
\]
If we choose $\tilde{t}_m : = (\delta,0,\dots,0) \in \bR^n \setminus \{o\}$, we get with the triangle inequality
\[
\forall x \in C \,\, ({\rm dist} (x,\bR^n \setminus ({\rm int} (K) + \tilde{t}_m)) > 0 )\,,
\]
i.e.,
\[
C \subseteq {\rm int}(K) + \tilde{t}_m
\]
with $\tilde{t}_m \not= o$. Thus
\[
K \subseteq \left(\bigcup \limits^{m-1}_{i=1} {\rm int} (K) + t_i\right) \cup \left({\rm int} (K) + \tilde{t}_m\right),
\]
where the vectors $t_1,\dots,t_{m-1}, \tilde{t}_m$ are different from $o$. This
finishes our note.

We continue the proof  of (ii) with
\[
\begin{array}{rcl}
c(K) &\!=\!& \min \left\{m : \exists t_1,\dots,t_m \in \bR^n \setminus \{o\} \left(K \subseteq \bigcup \limits^m_{i=1} {\rm int}(K) + t_i \right)\right\}\\[2ex]
&\!=\!& \min \left\{m : \exists t_1,\dots,t_m \in \bR^n \setminus \{o\} \, \forall x \in K\,\, \exists i \left(x \in {\rm int}(K) + t_i\right)\right\}\\[1ex]
&\!=\!& \min \left\{ m : \exists y_1,\dots,y_m \in \bR^n \setminus \{o\} \, \forall x \in K\,\, \exists i \left(x+y_i \in {\rm int} (K)\right)\right\}\\[1ex]
&\!=\!& \min \left\{m : \exists r_1,\dots,r_m \in S^{n-1}\, \exists \varepsilon > 0 \, \forall x \in K\,\exists i \left(x+\varepsilon r_i \in {\rm int} (K)\right)\right\}\\[1ex]
&\!=\!& \min \limits_{\varepsilon > 0} c(K, \varepsilon)\,.
\end{array}
\]

(iii) From (i) we know that $c(K,\varepsilon) < \infty$ implies $i (K,\varepsilon) < \infty$.
Hence,
\[
\sup \{\varepsilon > 0 : c (K,\varepsilon) < \infty\} \le \sup \{\varepsilon > 0 : i (K,\varepsilon) < \infty\}\,.
\]
It remains to prove that
\[
\sup \{ \varepsilon > 0 : i (K,\varepsilon) < \infty\} \le R (K) \le \sup \{\varepsilon > 0:  c(K,\varepsilon) < \infty\}\,.
\]

For showing
\[
\sup \{\varepsilon > 0 : i (K,\varepsilon) < \infty \} \le R (K)\,,
\]
we consider $\varepsilon > 0$ with $i(K,\varepsilon) < \infty$. This means
\[
\exists m \,\, \exists r_1,\dots,r_m \in S^{n-1} \,\, \forall x \in K\,\, \exists i\,\, (x + \varepsilon r_i \in K)\,.
\]
Using $x = c_K \in K$ (see (\ref{eq_cK})) and $K \subseteq B(c_K,R(K))$, we find an index $i$ such that
$c_K + \varepsilon r_i \in B (c_K, R(K))$, and this implies $\varepsilon \le R(K)$.

Therefore $\sup \{\varepsilon > 0 : i(K,\varepsilon) < \infty\} \le R (K)$.

For the proof of
\[
R(K) \le \sup \{\varepsilon > 0: c(K,\varepsilon) < \infty\}\,,
\]
let $\varepsilon < R(K)$. We have to show that $c(K,\varepsilon) < \infty$, and we use the \emph{auxiliary statement}
\[
\forall x \in K\,\, \exists y (x) \in {\rm int} \, (K)\,\, (\|x-y(x)\| \ge \varepsilon)\,.
\]

\emph{Assumption}:
$\exists x_0 \in K \,\, \forall y \in {\rm int}\,(K)\,\, \left(\|x_0-y\| < \varepsilon\right)\,$.
Then $\|x_0-y\| \le \varepsilon$ for all $y \in K$ and in turn $K \subseteq B(x_0,\varepsilon)$, a contradiction to $\varepsilon < R(K)$\,.

We complete the proof of $c(K,\varepsilon) < \infty$ with a series of
implications, starting with a consequence of the auxiliary statement:

\[
\begin{array}{cl}
&\forall x \in K \,\,\exists r(x) : = \frac{y(x)-x}{\|y(x)-x\|} \in S^{n-1}\,\, \exists \varepsilon (x):=\| x-y(x) \| \ge \varepsilon \\[1ex]
&\left(x+\varepsilon(x) r (x) = y(x) \in {\rm int} (K)\right)\\
\stackrel{\varepsilon(x) \ge \varepsilon}{\Rightarrow} & \forall x \in K\,\, \exists r(x) \in S^{n-1}\, \left(x+\varepsilon r(x) \in
{\rm int}\,(K)\right)\\[1ex]
\Rightarrow & K \subseteq \bigcup \limits_{x \in K} {\rm int}\,(K) - \varepsilon r (x)\\
\stackrel{K \mbox{\scriptsize~is compact}}{{\!\!\!}\Rightarrow} & \exists m \,\,\exists x_1,\dots,x_m \in K \, \left(K \subseteq \bigcup \limits^m_{i=1} {\rm int}\,(K) - \varepsilon r (x_i)\right)\\[2ex]
\Rightarrow & \exists m \,\, \exists r_1,\dots,r_m \in S^{n-1} \left(K \subseteq \bigcup \limits^m_{i=1} {\rm int}\,(K) - \varepsilon r_i\right)\\[2ex]
\Rightarrow & \forall x \in K \,\,\exists i \in \{1,\dots,m\}\, \left(x \in {\rm int}\,(K) - \varepsilon r_i\right)\\[1ex]
\Rightarrow & \forall x \in K \,\, \exists  i \in \{1,\dots,m\}\, \left(x + \varepsilon r_i \in {\rm int}\,(K)\right)\\[1ex]
\Rightarrow & c(K,\varepsilon) \le m < \infty\,.
\end{array}
\]
\end{proof}

We study the behavior and mutual relations of the functions
$i(K,\cdot), c(K,\cdot) : (0,\infty) \to \bN \cup \{\infty\}$.
So far we know:
\vspace*{-0.3cm}
\begin{enumerate}
\item[(a)] $i(K,\cdot) \le c(K,\cdot)$.
\item[(b)] $i(K,\cdot), c(K,\cdot)$ are monotonically increasing.
\item[(c)] $i(K) \le t(K) = \lim \limits_{\varepsilon \downarrow 0} i(K,\varepsilon) \le \lim \limits_{\varepsilon \downarrow 0} c(K,\varepsilon) = c(K)$,
where we have examples with ``$\not=$'' in both ``$\le$'' estimates.
\item[(d)] $i(K,\varepsilon) \le c(K,\varepsilon) < \infty$ for  $\varepsilon < R(K)$ and $\infty = i(K,\varepsilon) = c(K,\varepsilon)$ for $\varepsilon > R(K)$.
\end{enumerate}
Further facts about $i(K,\cdot)$ and  $c(K,\cdot)$ are given in the following

\begin{theorem}\label{theo3.2}
The following properties are satisfied for every convex body $K \subseteq \mathbb{R}^n$.
\begin{enumerate}
\item[(e)] $i(K,\cdot)$ is left-continuous, i.e., for every $\varepsilon > 0$,
\[
i (K,\varepsilon) = \lim \limits_{\varepsilon' \uparrow \varepsilon} i(K,\varepsilon') \stackrel{(b)}{=} \sup \{i(K,\varepsilon') : 0 < \varepsilon' < \varepsilon\}\,.
\]
\item[(f)] $c(K,\cdot)$ is right-continuous, i.e., for every $\varepsilon > 0$,
\[
c(K,\varepsilon) = \lim \limits_{\varepsilon' \downarrow \varepsilon} c(K,\varepsilon') \stackrel{(b)}{=} \inf \{c(K,\varepsilon') : \varepsilon' > \varepsilon\}\,.
\]
\end{enumerate}
\end{theorem}

\begin{proof}
For (e), we first consider the case $\sup \{i(K,\varepsilon') : 0 < \varepsilon' < \varepsilon\} = \infty$. Here we obtain $i(K,\varepsilon) = \infty$ by (b).

Now we suppose that $\sup \{i(K,\varepsilon') : 0 < \varepsilon' < \varepsilon\} = m < \infty$.
Property (b) implies
\[
\exists \varepsilon_0 \in (0,\varepsilon) \,\, \forall \varepsilon' \in (\varepsilon_0,\varepsilon)\,\, (i(K,\varepsilon') = m)\,,
\]
and $i(K,\varepsilon) \le m$ has to be proved.

Consider the sequence $\varepsilon_k' : = \varepsilon - \frac{1}{k}$ for $k \ge k_0$ (such that $\varepsilon_k' \in (\varepsilon_0,\varepsilon)$).
Then $i(K,\varepsilon'_k) = m$ is given by the illumination vectors $r^{(k)}_1, \dots, r^{(k)}_m \in S^{n-1}$. We can assume that $r^{(k)}_1 \stackrel{k \to \infty}{\longrightarrow} r_1,\dots, r^{(k)}_m \stackrel{k \to \infty}{\longrightarrow} r_m$, because $S^{n-1}$ is compact.

\emph{Assertion}: The vectors $r_1,\dots,r_m \in S^{n-1}$ assure $i(K,\varepsilon) \le m$.

To see this, let $x \in K$. We have to prove that $x + \varepsilon r_i \in K$ for some
$i \in \{1,\dots,m\}$. The inclusion $x \in K$ implies that
$$
\forall k \ge k_0\,\, \exists i_k \in \{1,\dots,m\} \,\, \left(x + \varepsilon'_k  r^{(k)}_{i_k} \in K \right)
$$
and, therefore,
$(i_k)^\infty_{k=1}$ has a constant subsequence $(i_{k_l})^\infty_{l=1}, \,\, i_{k_l} \equiv i_0$, where
\[
 \forall l\, \left(K \ni x + \varepsilon'_{k_l}  r_{i_{k_l}}^{(k_l)} = x + \left(\varepsilon - \frac{1}{k_l}\right)
r^{(k_l)}_{i_0}\right)\,.
\]
The latter two terms tend to $\varepsilon$ and $r_{i_0}$, respectively, as $l \to \infty$. Since $K$ is closed, we get $x + \varepsilon r_{i_0} \in K$, and (e) is verified.

For (f), we have to show the implication
\[
c(K,\varepsilon) = m < \infty \quad\Rightarrow\quad \exists \varepsilon'_0 > \varepsilon\, \left(c(K,\varepsilon'_0) \le m \right)\,.
\]
To see this, let $c(K,\varepsilon) = m$ be assured by $r_1,\dots,r_m \in S^{n-1}$. Then,
for every $x \in K$, there is $i \in \{1,\dots,m\}$ such that $\left(x + \varepsilon r_i \in {\rm int} (K)\right)$. Hence
\[
K \subseteq
\bigcup \limits^m_{i=1} {\rm int}\,(K) - \varepsilon r_i
\subseteq
\bigcup \limits^m_{i=1}
\bigcup \limits_{\varepsilon' > \varepsilon} {\rm int}\, (K) - \varepsilon' r_i\,.
\]
Since $K$ is compact, we have a finite subcover
\begin{equation}
\label{eq1}
K \subseteq \bigcup \limits^m_{i=1} \bigcup \limits_{\varepsilon' \in \left\{\varepsilon^{(i)}_1, \dots, \varepsilon^{(i)}_{k_i}\right\}} {\rm int}\,(K) - \varepsilon' r_i\,.
\end{equation}
We choose $\varepsilon'_0$ with $\varepsilon < \varepsilon'_0 \le \min \left(\left\{\varepsilon^{(1)}_1,\dots,\varepsilon^{(1)}_{k_1}\right\} \cup \dots \cup \left\{\varepsilon^{(m)}_1,\dots,\varepsilon^{(m)}_{k_m}\right\}\right)$ and claim that
\begin{equation}
\label{eq2}
\forall x \in K\,\, \exists i \in \{1,\dots,m\}\,\, (x + \varepsilon'_0 r_i \in {\rm int}\,(K))\,.
\end{equation}
Indeed, if $x \in K$ satisfies $x \in {\rm int}\,(K)-\varepsilon' r_i$ for some
$\varepsilon' \ge \varepsilon'_0$, then $x+\varepsilon' r_i \in {\rm int}\,(K)$, and this yields
$x+ \varepsilon'_0 r_i \in {\rm int}\,(K)$. Therefore (\ref{eq1}) implies (\ref{eq2}).

Property (\ref{eq2}) gives $c(K,\varepsilon'_0) \le m$, where $\varepsilon'_0 > \varepsilon$, and the proof is complete.
\end{proof}

We continue with looking at the behavior of $i(K,\cdot)$ and $c(K,\cdot)$ close to the critical argument $R(K)$.


\section{The behavior of $i(K,\cdot)$ and $c(K,\cdot)$ at $R(K)$}

From (f) and (d) we obtain
\begin{enumerate}
\item[(g)] $c(K,R(K)) = \infty$.
\end{enumerate}
The study of $i(K,R(K))=\sup\limits_{0 < \varrho < R(K)} i(K,\varrho)$ and of $\sup\limits_{0 < \varrho < R(K)} c(K,\varrho)$ is more involved.

\begin{theorem}\label{theo4.1}
The following are equivalent for every convex body $K \subseteq \mathbb{R}^n$.
\begin{enumerate}
\item[(i)] $i(K,R(K)) < \infty$.
\item[(ii)] $\sup\limits_{0 < \varrho < R(K)} c(K,\varrho) < \infty$.
\item[(iii)] There is a finite set
\[
\{x_1,\dots,x_k\} \subseteq K \cap {\rm bd}(B(c_K,R(K)))
\]
such that $\bigcup \limits^k_{i=1} K + (c_K-x_i)$ covers a neighborhood of $c_K$ in $K$, i.e.,
\[
\exists \delta > 0\, \left(K \cap B^o (c_K,\delta) \subseteq \bigcup\limits^k_{i=1} K + (c_K - x_i)\right)\,.
\]
\item[(iv)] There is a finite set
\[
\{x_1,\dots,x_k\} \subseteq K \cap {\rm bd}(B(c_K,R(K)))
\]
such that $\bigcup \limits^k_{i=1} K + (c_K-x_i)$ covers a neighborhood of $c_K$, i.e.,
\[
\exists \delta > 0\, \left( B^o (c_K,\delta) \subseteq \bigcup\limits^k_{i=1} K + (c_K - x_i)\right)\,.
\]
\end{enumerate}
\end{theorem}

Examples from the following section will show that $i(K,R(K))$ can be finite, but need not
be.

\begin{proof}[Proof of Theorem~\ref{theo4.1}, $(ii) \Rightarrow (i)$]
We estimate
\[
i(K,R(K)) \stackrel{{\rm (e)}}{=} \sup\limits_{0<\varrho<R(K)} i(K,\varrho) \stackrel{{\rm (a)}}{\le} \sup\limits_{0<\varrho<R(K)} c(K,\varrho) \stackrel{{\rm (ii)}}{<} \infty\,.
\]
\end{proof}

\begin{proof}[Proof of Theorem~\ref{theo4.1}, $(i) \Rightarrow (iii)$]
Let $i (K, R(K)) = m < \infty$. Then there exist $r_1,\dots,r_m \in S^{n-1}$ such that every $x \in K$ satisfies $x+R(K) r_i \in K$ for some $i$, i.e.,
\[
K \subseteq \bigcup \limits^m_{i=1} K-R(K) r_i\,.
\]

By (\ref{eq_cK}), $c_K \in K$; w.l.o.g.,
$c_K \in K - R (K) r_i$, $1 \le i \le k$, and
$c_K \notin K-R (K) r_i$, $k < i \le m$ ($k \in \{1,\ldots,m\}$).
Thus,
\[
U:= K \setminus \bigcup \limits^m_{i=k+1} K - R (K) r_i\,
\]
is a neighborhood of $c_K$ with respect to the relative topology of $K$, and
\[
U \subseteq \bigcup \limits^k_{i=1} K - R (K) r_i \,.
\]
The inclusions $c_K \in K - R (K)r_i$, $1 \le i \le k$, yield $c_K + R(K)r_i \in K$,
where $c_K+R(K)r_i \in {\rm bd}\,(B(c_K,R(K)))$. Therefore
\[
x_i:= c_K + R (K)r_i \in K \cap {\rm bd}\,(B(c_K,R(K)))
\]
for $i \in \{1,\dots,k\}$. Thus we have
\[
 U \subseteq \bigcup \limits^k_{i=1} K-R(K)r_i = \bigcup \limits^k_{i=1} K+(c_K-x_i)\,.
\]
\end{proof}

\begin{lemma}
\label{lem4.1}
Let $K \subseteq \bR^n$ be a convex body such that the origin $o$ belongs to the boundary ${\rm bd}\,(K)$. Then there exists $x_0 \in {\rm int}\,(K)$ such that $\langle x,x_0 \rangle \ge 0$ for all $x \in K$, where $\langle \cdot,\cdot \rangle$ denotes the standard inner product in $\bR^n$.
\end{lemma}

\begin{proof}
Define
\[
L:= \{y \in \bR^n: \forall x \in {\rm int}\, (K)\,\,(\langle x,y \rangle > 0)\}\,.
\]
It is enough to show that $L$ has a common point $x_0$ with ${\rm int}\,(K)$. Suppose that $L \cap {\rm int}\,(K) = \emptyset$.

Of course, $L$ is convex. $o$ can be separated from ${\rm int}\,(K)$ by some hyperplane. Then one normal vector of that hyperplane belongs to $L$. Hence, $L \ne \emptyset$. Moreover, $o$ belongs to the closure ${\rm cl}\,(L)$, because $y \in L$ implies $\lambda y \in L$ for every $\lambda > 0$.

By $L \cap {\rm int}\,(K)= \emptyset$, both sets can be separated: there are $y_0 \in \bR^n \setminus \{o\}$ and $c \in \bR$ such that
\begin{equation}
\label{eq-lem1}
\forall y \in L\,\,( \langle y,y_0 \rangle \le c ) \quad\mbox{ and }\quad \forall x \in {\rm int}\,(K)\,\,
( \langle x,y_0 \rangle > c )\,.
\end{equation}
Note that $c=0$, because $o \in {\rm cl}\,(L)$ and $o \in {\rm cl}\,({\rm int}\, (K))$. Now the right-hand part of (\ref{eq-lem1}) yields $y_0 \in L$, and then the left-hand property gives $\langle y_0,y_0 \rangle \le 0$, a contradiction.
\end{proof}

\begin{lemma}
\label{lem4.2}
Let $K \subseteq \bR^n$ be a convex body such that $c_K \notin {\rm int}\,(K)$. Then
\[
\forall \delta > 0\,\, \left( K \cap B^o(c_K,\delta) \not\subseteq \bigcup\limits_{x \in K \cap {\rm bd}\,(B(c_K,R(K)))} K+(c_K-x) \right)\,.
\]
In particular, $K$ does not satisfy condition (iii) from Theorem~\ref{theo4.1}.
\end{lemma}

\begin{proof}
We can assume that $B(c_K,R(K))=B(o,1)$. By (\ref{eq_cK}),
we obtain $o=c_K \in {\rm bd}\,(K)$. Then Lemma~\ref{lem4.1} gives $x_0 \in {\rm int}\,(K)$ such that
\begin{equation}
\label{eq-lem2}
\forall x \in K\,\,(\langle x,x_0 \rangle \ge 0)\,.
\end{equation}

Now suppose that the claim of Lemma~\ref{lem4.2} is false, i.e.,
\[
K \cap B^o(o,\delta) \subseteq \bigcup\limits_{x \in K \cap {\rm bd}\,(B(c_K,R(K)))} K-x
\]
for some $\delta > 0$. Since the half-line $\{\lambda x_0: \lambda > 0\}$ emanates from $o \in K$ and passes through $x_0 \in K$, there exists $\lambda_0 > 0$ such that $\lambda_0 x_0 \in K \cap B^o(o,\delta)$. Thus there is $x_1 \in K \cap {\rm bd}\,(B(c_K,R(K)))$ such that $\lambda_0 x_0 \in K-x_1$. This yields
\[
\lambda_0 x_0 + x_1 \in K \subseteq B(c_K,R(K))=B(o,1)\,,
\]
and therefore, by the aid of $\lambda_0\|x_0\|>0$, $\|x_1\|=1$, and (\ref{eq-lem2}),
\[
1 \ge \|\lambda_0 x_0 +x_1 \|^2=
\lambda_0^2\|x_0\|^2+2\lambda_0\langle x_1,x_0 \rangle+ \|x_1\|^2 > 0+0+1\,,
\]
a contradiction.
\end{proof}

\begin{proof}[Proof of Theorem~\ref{theo4.1}, $(iii) \Rightarrow (iv)$]
By (iii) and Lemma~\ref{lem4.2}, $K$ contains a neighborhood of $c_K$. Therefore (iii) implies (iv).
\end{proof}

\begin{proof}[Proof of Theorem~\ref{theo4.1}, $(iv) \Rightarrow (ii)$]
Condition (iv) implies (iii), and therefore Lemma~\ref{lem4.2} yields $c_K \in {\rm int}\,(K)$. Thus we can assume that $\delta$ in condition (iv) is small enough such that $B^o(c_K,\delta) \subseteq {\rm int}\,(K)$.
Hence we are given
\[
\{x_1,\ldots,x_k\} \subseteq K \cap {\rm bd}\,(B(c_K,R(K)))
\]
and $\delta > 0$ such that
\[
B^o(c_K,\delta) \subseteq {\rm int}\,(K) \cap \bigcup\limits_{i=1}^k K+(c_K-x_i)\,.
\]
Since $\|x_i-c_K\| = R(K)$, we have $r_i:=\frac{x_i-c_K}{R(K)} \in S^{n-1}$ and
\[
B^o (c_K,\delta) \subseteq {\rm int}\,(K) \cap \bigcup\limits^k_{i=1} K-R (K) r_i\,.
\]
Pick an arbitrary $x \in B^o(c_K,\delta)$. Then $x \in {\rm int}\,(K)$, and there exists $i \in \{1,\ldots,k\}$ such that $x \in K-R(K)r_i$, i.e., $x+R(K)r_i \in K$. Hence, for every $\varrho \in (0,R(K))$, $x + \varrho r_i=
\frac{R(K)-\varrho}{R(K)}x+ \frac{\varrho}{R(K)}(x+R(K)r_i) \in {\rm int}\,(K)$. We obtain
\begin{equation}
\label{*}
\forall x \in B^o(c_K,\delta)\,\, \exists i \in \{1,\ldots,k\}\,\, \forall \varrho \in (0,R(K))\,\,(x+\varrho r_i \in {\rm int}\,(K))\,.
\end{equation}

Next we consider points from $K \setminus B^o(c_K,\delta)$. Note that
\[
\forall x \in K \setminus B^o(c_K,\delta)\,\, \exists z (x) \in {\rm int}\,(K) \,\, \left(\|x-z(x)\|=R(K)\right)\,.
\]
(This holds, since otherwise there existed $x \in K \setminus B^o(c_K,\delta)$ such that every $z \in {\rm int}\, (K)$ would
satisfy $\|x-z\| < R(K)$, hence every $z \in K$ satisfied $\|x-z\| \le R(K))$,
and so $K \subseteq B(x,R(K))$. Then $x = c_K$, contradicting $x \notin B^o (c_K,\delta)$.)
We obtain $x =z(x)+(x-z(x)) \in {\rm int}\,(K)+(x-z(x))$.
Hence
\[
 K \setminus B^o (c_K,\delta) \subseteq \bigcup \limits_{x \in K \setminus B^o(c_K,\delta)} \, {\rm int}\,(K) + (x-z(x))
\]
and, since $K \setminus B^o(c_K,\delta)$ is compact, there exist finitely many
vectors $x_{k+1}$, $\dots$, $x_m \in K \setminus B^o(c_K,\delta)$ such that
\[
K \setminus B^o(c_K,\delta) \subseteq \bigcup \limits^m_{i=k+1}
{\rm int}\,(K) + (x_i - z(x_i))\,.
\]
Now we set $r_i:= \frac{z (x_i) - x_i}{R(K)} \in S^{n-1}$, $k+1 \le i \le m$.
We get
\[
  K \setminus B^o (c_K,\delta) \subseteq \bigcup \limits^m_{i=k+1} {\rm int}\, (K) - R(K) r_i\,.
\]
That is, for every $x \in K \setminus B^o(c_K,R(K))$, there exists $i \in \{k+1,\ldots,m\}$ such that $x+R(K)r_i \in {\rm int}\,(K)$. This yields $x+\varrho r_i= \frac{R(K)-\varrho}{R(K)}x+ \frac{\varrho}{R(K)}(x+R(K)r_i) \in {\rm int}\,(K)$ for $0 < \varrho < R(K)$. Hence
\[
\forall x \in K \setminus B^o(c_K,\delta)\,\, \exists i \in \{k+1,\ldots,m\}\,\, \forall \varrho \in (0,R(K))\,\,(x+\varrho r_i \in {\rm int}\,(K))\,.
\]

The last property and (\ref{*}) show that $c(K,\varrho) \le m$ for $0 < \varrho < R(K)$. This proves (ii).
\end{proof}

Theorem~\ref{theo4.1} implies several necessary conditions.

\begin{corollary}
\label{cor4.1}
If a convex body $K \subseteq \bR^n$ satisfies conditions (i)-(iv) from Theorem~\ref{theo4.1}, then
\begin{enumerate}
\item[($\alpha$)] every open half-sphere of ${\rm bd}\, (B(c_K,R(K)))$ contains an element of $K$,
\item[($\beta$)] $c_K \in {\rm int}\,(K)$,
\item[($\gamma$)] the affine hull of $K \cap {\rm bd}\, (B(c_K,R(K)))$ is $\bR^n$, and
\item[($\delta$)] $K \cap {\rm bd}\, (B(c_K,R(K)))$ contains at least $n+1$ points.
\end{enumerate}
\end{corollary}

\begin{proof}
Without loss of generality, $B(c_K,R(K))=B(o,1)$. In particular, ${\rm bd}\,(B(c_K,R(K)))=S^{n-1}$.

Let us assume that ($\alpha$) fails. Then there exists $x_0 \in \bR^n \setminus \{o\}$ such that the open half-sphere $\{x \in S^{n-1}: \langle x,x_0 \rangle < 0 \}$ does not meet $K$,
\[
\forall x \in K \cap S^{n-1} \,\,( \langle x,x_0 \rangle \ge 0 )\,.
\]
For sufficiently small $\lambda_0 > 0$, the point $\lambda_0 x_0$ belongs to the neighborhood $B^o(o,\delta)$ of $o=c_K$, that is covered by the sets $K-x_i= K+(c_K-x_i)$, $1 \le i \le k$, according to (iv). Hence
$
\lambda_0 x_0 \in K-x_i
$
for some $x_i \in \{x_1,\ldots,x_k\} \subseteq K \cap S^{n-1}$. We obtain
\[
\lambda_0 x_0 + x_i \in K \subseteq B(o,1)
\]
and
\[
1 \ge \|\lambda_0 x_0+x_i\|^2=\lambda_0^2\|x_0\|^2+2\lambda_0\langle x_i,x_0 \rangle+ \|x_i\|^2 > 0+0+1\,,
\]
because $\lambda_0 > 0$, $x_0 \ne o$, $\langle x_i,x_0 \rangle \ge 0$, and
$\|x_i\|=1$. This contradiction proves ($\alpha$).

Condition ($\beta$) is a consequence of Lemma~\ref{lem4.2}. (Alternatively, it follows directly from ($\alpha$) by a separation argument.)

If ($\gamma$) failed, then $K \cap {\rm bd}\,(B(c_K,R(K)))$ would be contained in a hyperplane and, therefore, ($\alpha$) would fail as well.

Of course, ($\gamma$) implies ($\delta$).
\end{proof}

\begin{corollary}
\label{cor4.2}
If $K \subseteq \bR^n$ is an $n$-dimensional convex polytope with vertex set ${\rm vert}\,(K)$, then the set
\[
K \cap {\rm bd}\,(B(c_K,R(K)))= {\rm vert}\,(K) \cap {\rm bd}\,(B(c_K,R(K)))
\]
is finite, and therefore properties (i)-(iv) from Theorem~\ref{theo4.1} are equivalent to
\[
\exists \delta > 0\,\, \left(K \cap B^o (c_K,\delta) \subseteq \bigcup\limits_{v \in {\rm vert}\,(K) \cap {\rm bd}\,(B(c_K,R(K)))} K + (c_K - v)\right),
\]
as well as to
\[
\exists \delta > 0\,\, \left(B^o (c_K,\delta) \subseteq \bigcup\limits_{v \in {\rm vert}\,(K) \cap {\rm bd}\,(B(c_K,R(K)))} K + (c_K - v)\right)\,.
\]
\end{corollary}

In the two-dimensional setting, condition (iii) from Theorem~\ref{theo4.1} can be formally weakened.

\begin{theorem}
\label{theo4.2}
For a two-dimensional convex body $K \subseteq \bR^2$ the following is equivalent to conditions (i)-(iv) from Theorem~\ref{theo4.1}.
\begin{enumerate}
\item[(v)] The set $\bigcup \limits_{x \in K \cap {\rm bd}\,(B(c_K,R(K)))} K + (c_K-x)$ covers a neighborhood of $c_K$ in $K$, i.e.,
\[
\exists \delta > 0\,\, \left(K \cap B^o (c_K,\delta) \subseteq \bigcup\limits_{x \in K \cap {\rm bd}\,(B(c_K,R(K)))} K + (c_K - x)\right)\,.
\]
\end{enumerate}
\end{theorem}

\begin{proof}
The implication (iii)$\Rightarrow$(v) is trivial.

Now we suppose (v) and shall show (iv). Without loss of generality, we assume $B(c_K,R(K))=B(o,1)$, in particular, ${\rm bd}\,(B(c_K,R(K)))=S^1$.
By Lemma~\ref{lem4.2}, $o=c_K \in {\rm int}\, (K)$. Then (v) gives some $\delta > 0$ such that
\begin{equation}
\label{eq-(v)1}
B(o,\delta) \subseteq \bigcup\limits_{x \in K \cap S^1} K - x\,.
\end{equation}
%
%

We decompose the circle ${\rm bd}\,(B(c_K,R(K)))=S^1$ into four arcs $\Gamma_i$, $1 \le i \le 4$, each representing an angle of size $\frac{\pi}{2}$. With every $\Gamma_i$ we associate a set $H_i \subseteq K \cap \Gamma_i$, namely
\begin{enumerate}
\item $H_i:= K \cap \Gamma_i$ if $|K \cap \Gamma_i| \le 2$,
\item $H_i$ consists of three distinct points $a_i,b_i,c_i \in K \cap \Gamma_i$, where $a_i$ and $b_i$ are those elements of $K \cap \Gamma_i$ that are closest to the end-points of $\Gamma_i$ and $c_i$ is chosen arbitrarily
in $(K \cap \Gamma_i) \setminus\{a_i,b_i\}$, if $|K \cap \Gamma_i| \ge 3$.
\end{enumerate}
We put $H=H_1\cup H_2 \cup H_3 \cup H_4$ and want to prove that
\begin{equation}
\label{eq-(v)2}
\forall r \in S^1\,\, \forall \nu \in (0,\nu_0]\,\, \exists x_0 \in H\,\,
(\nu r + x_0 \in K)\,,
\end{equation}
where
\[
\nu_0 =
\min\left\{\frac{\delta}{2}, \min\big\{ {\rm dist}\,(c_i,{\rm conv}\,\{o,a_i\} \cup {\rm conv}\,\{o,b_i\}): |H_i|=3 \big\} \right\}.
\]
Here $\min \emptyset= \infty$ and ${\rm dist}(a,B)= \inf\{\|a-b\|: b \in B\}$. Clearly, (\ref{eq-(v)2}) implies
\[
B(o,\nu_0) \subseteq \bigcup\limits_{x \in H} K-x \,,
\]
which in turn gives (iv) with the additional bound $k \le 12$.

For the proof of (\ref{eq-(v)2}), let $r \in S^1$ and $\nu \in (0,\nu_0]$ be fixed. By (\ref{eq-(v)1}), there exists $x_r \in K \cap S^1$ such that $\delta r \in K-x_r$. That is,
\[
\delta r + x_r \in K\,.
\]

\emph{Case 1: $x_r \in H$. } Putting $x_0:=x_r$ and using $x_r, \delta r+x_r \in K$ as well as $0 \le \nu \le \nu_0 \le \delta$, we obtain claim (\ref{eq-(v)2}) by
\[
\nu r+x_0= \nu r+x_r \in {\rm conv}\,\{x_r,\delta r+x_r\} \subseteq K\,.
\]

\emph{Case 2: $x_r \notin H$. } By the construction of $H_1,\dots,H_4$, there exists $i \in \{1,2,3,4\}$ such that $x_r$ is located on the arc $\Gamma_i$ strictly between the extremal points $a_i,b_i \in H_i$. Note that we have
\begin{equation}
\label{eq-(v)3}
o,a_i,b_i,c_i,x_r,x_r+\delta r \in K\,.
\end{equation}
For further computations and illustrations we assume $r=(-1,0)$. Then $x_r+\delta r \in K \subseteq B(o,1)$ implies that the first coordinate of $x_r$ is positive.

{\emph{Subcase 2.1: $o$ is in the open slab bounded by $\{a_i+\mu r: \mu \in \bR\}$ and $\{b_i+\mu r: \mu \in \bR\}$} (see Figure~\ref{fig-(v)}). }
\begin{figure}
\begin{center}
\begin{tikzpicture}[scale=2]
  \draw (0,0) node[above] {\tikz[scale=2]{
    \draw (0,0) circle (1) (.6,.8) -- (0,0) -- (.98,-.2) (0,-1) node[below] {\small $S^1={\rm bd}\,(B(c_K,R(K)))$} (-1,0) node[left] {\small $r$} (0,0) node[below left] {\small $o$} (.98,-.2) node[below right] {\small $a_i$} (.6,.8) node[above right] {\small $b_i$} (.74,.68) node[right] {\small $x_r$} (.5,.2) node[below] {\small $x_0+\nu r$} (.98,.2) node[right] {\small $c_i=x_0$} (-.3,.71) node[below] {\small $x_r+\delta r$} (.15,.2) node[above left] {\small $s$};  
    \fill (-1,0) circle (.03) (0,0) circle (.03) (.98,-.2) circle (.03) (.6,.8) circle (.03) (.98,.2) circle (.03) (.5,.2) circle (.03) (.71,.71) circle (.03) (-.3,.71) circle (.03) (.15,.2) circle (.03);
    \draw[densely dotted] (-1,0) -- (0,0) (-1.15,-.2) -- (1.15,-.2) (-1.15,.8) -- (1.15,.8) (.15,.2) -- (.98,.2) (-.3,.71) -- (.71,.71);
  }};
  \draw (3,0) node[above] {\tikz[scale=2]{
    \draw (0,0) circle (1) (.21,.71) -- (0,0) (0,-1) node[below] {\small $S^1={\rm bd}\,(B(c_K,R(K)))$} (-1,0) node[left] {\small $r$} (0,0) node[below] {\small $o$} (.6,.8) (.71,.71) node[above right] {\small $x_r$} (.5,.2) node[below] {\small $x_0+\nu r$} (.98,.2) node[above right] {\small $a_i=x_0$} (-.29,.71) node[below] {\small $x_r+\delta r$} (-.1,.995) node[above] {\small $b_i$} (.21,.71) node[above] {\small $p$} (.06,.2) node[above left] {\small $q$};  
    \fill (-1,0) circle (.03) (0,0) circle (.03) (.98,.2) circle (.03) (.5,.2) circle (.03) (.71,.71) circle (.03) (-.29,.71) circle (.03) (.21,.71) circle (.03) (.06,.2) circle (.03) (-.1,.995) circle (.03);
    \draw[densely dotted] (-1,0) -- (0,0) (-1.15,.2) -- (1.15,.2) (-1.15,.995) -- (1.15,.995) (-.3,.71) -- (.71,.71);
  }};
\end{tikzpicture}
\end{center}
\caption{Subcases 2.1 (on the left) and 2.2 (on the right)}
\label{fig-(v)}
\end{figure}
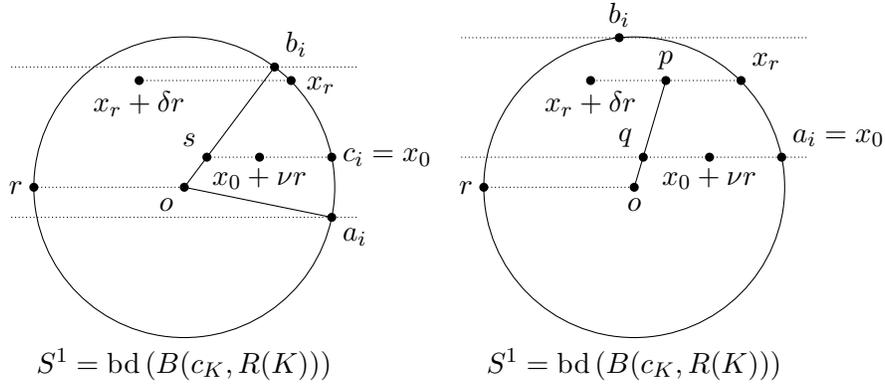
We put $x_0:=c_i \in H_i$. Hence $x_0$ is on $\Gamma_i$ strictly between $a_i$ and $b_i$. Since $\nu \le \nu_0 \le {\rm dist}\,(c_i,{\rm conv}\,\{o,a_i\} \cup
{\rm conv}\,\{o,b_i\})$, $x_0+\nu r$ is on the segment between $x_0=c_i$ and the intersection point $s$ of ${\rm conv}\,\{o,a_i\} \cup
{\rm conv}\,\{o,b_i\}$ with the straight line $\{c_i+\mu r: \mu \in \bR\}$.
Now (\ref{eq-(v)3}) gives the claim $x_0+\nu r \in K$.

{\emph{Subcase 2.2: $o$ is not in the open slab bounded by $\{a_i+\mu r: \mu \in \bR\}$ and $\{b_i+\mu r: \mu \in \bR\}$} (see Figure \ref{fig-(v)}). } Without loss of generality, $a_i=(\cos \alpha, \sin \alpha)$ and $x_r= (\cos \varphi, \sin \varphi)$ with $0 \le \alpha < \varphi < \frac{\pi}{2}$. We put $x_0:=a_i \in H_i$. By (\ref{eq-(v)3}), the midpoint $p$ of $x_r$ and $x_r+\delta r$ belongs to $K$ and has a non-negative first coordinate. Then the intersection point $q$ of the segment $o p$ and the straight line $\{a_i+ \mu r: \mu \in \bR\}$ belongs to $K$, too, and satisfies
\[
\|a_i-q\| > \|p-x_r\|= \left\|\frac{1}{2}\delta r\right\|=\frac{\delta}{2} \ge \nu_0 \ge \nu\,.
\]
Therefore $x_0+\nu r= a_i+\nu r$ belongs to the segment $a_i q$ and hence to $K$ as well. This shows (\ref{eq-(v)2}) and completes the proof of Theorem~\ref{theo4.2}.
\end{proof}

\begin{remark}\label{rem3.2}
\textbf{(A)} Note that the set $H=H_1 \cup H_2 \cup H_3 \cup H_4$ in the above proof contains at most $12$ points. Therefore conditions (iv) and (iii) from Theorem~\ref{theo4.1} can be sharpened by the additional restriction $k \le 12$, provided we are in the two-dimensional situation $n=2$.

\textbf{(B)} Condition (v) from Theorem~\ref{theo4.2} fails to be equivalent to (i)-(iv) from Theorem~\ref{theo4.1} as soon as the dimension $n$ exceeds $2$. This is illustrated by the compact double cone
\[
K:= {\rm conv}\left(\{(0,\dots,0,\pm 1)\} \cup \left\{(\xi_1,\dots,\xi_{n-1},0): \xi_1^2+\dots+\xi_{n-1}^2=1\right\}\right) \subseteq \bR^n,
\]
that is already mentioned after inequality (\ref{2}).
\end{remark}

\begin{proof}[Proof of \textbf{(B)}]
For a vector $x=(\xi_1,\dots,\xi_n) \in \bR^n$, we write $x=(x[<n],x[n])$ where $x[<n]=(\xi_1,\dots,\xi_{n-1})$ and $x[n]=\xi_n$. With this notation,
\[
K=\{x \in \bR^n: \|x[<n]\|+|x[n]| \le 1\},
\]
where, of course, $\|x[<n]\|$ is the Euclidean norm of $x[<n]$ in $\bR^{n-1}$.
Clearly,
$B(c_K,R(K))=B(o,1)$ and
\[
K \cap {\rm bd }(B(c_K,R(K)))= \{(0,\dots,0,\pm 1)\} \cup \left\{(y,0): y \in S^{n-2}\right\}\,.
\]

\emph{Part 1: $K$ satisfies (v). } We shall show that
\[
K \subseteq \bigcup\limits_{x \in \{(0,\dots,0,\pm 1)\} \cup \left\{(y,0): y \in S^{n-2}\right\}} K-x\,.
\]
Let $x_0 \in K$.

\emph{Case 1: $x_0=o$. } Of course, $o=(0,\dots,0,1)-(0,\dots,0,1) \in K - (0,\dots,0,1)$.

\emph{Case 2: $\|x_0[<n]\|<|x_0[n]|$. } The assumption $\|x_0[<n]\|<|x_0[n]|$ implies ${\rm sgn}\,(x_0[n]) \ne 0$ and
\[
\|x_0[<n]\|+|x_0[n]-{\rm sgn}\,(x_0[n])| = \|x_0[<n]\|+(1-|x_0[n]|) \le 1\,,
\]
which gives $(x_0[<n],x_0[n]-{\rm sgn}\,(x_0[n])) \in K$ and
\[
x_0 \in K-(0,\dots,0,-{\rm sgn}\,(x_0[n]))\,.
\]

\emph{Case 3: $\|x_0[<n]\| \ge |x_0[n]|$ and $x_0 \ne o$. } In this case the assumption gives $\|x_0[<n]\| \ne 0$ and
\[
\left\| x_0[<n]- \frac{x_0[<n]}{\|x_0[<n]\|} \right\|+ |x_0[n]| = (1-\|x_0[<n]\|)+|x_0[n]| \le 1\,,
\]
yielding $\left( x_0[<n]- \frac{x_0[<n]}{\|x_0[<n]\|},x_0[n] \right) \in K$ and
\[
x_0 \in K-\left(-\frac{x_0[<n]}{\|x_0[<n]\|},0 \right)\,.
\]

\emph{Part 2: $K$ does not satisfy $(iv)$. } Let
\[
\{x_1,\dots,x_k\} \subseteq K \cap {\rm bd }(B(c_K,R(K)))= \{(0,\dots,0,\pm 1)\} \cup \left\{(y,0): y \in S^{n-2}\right\}
\]
be a finite set. We pick $x_0 \in \left\{(y,0): y \in S^{n-2}\right\} \setminus \{x_1,\dots,x_k\}$ and define
\[
c:= \max(\{\langle x_0[<n], x_i[<n] \rangle: 1 \le i \le k\} \cup \{0\}) \in [0,1)\,,
\]
where $\langle \cdot,\cdot \rangle$ is the standard inner product in $\bR^{n-1}$. For disproving (iv) it is enough to show that the set $\bigcup\limits_{i=1}^k K-x_i=\bigcup\limits_{i=1}^k K+(c_K-x_i)$ and the half-line $\{\mu(-x_0[<n],c): \mu > 0\}$ emanating from $o=c_K$ are disjoint. That is,
\begin{equation}
\label{eqx}
\forall \mu > 0\,\, \forall i \in \{1,\dots,k\}\,\,((-\mu x_0[<n],\mu c)+(x_i[<n],x_i[n]) \notin K )\,.
\end{equation}
To obtain a contradiction, suppose that (\ref{eqx}) fails for a particular choice of $\mu$ and $i$. Then $(-\mu x_0[<n]+x_i[<n],\mu c+x_i[n]) \in K$, i.e.,
\begin{equation}
\label{eq-part2}
\|-\mu x_0[<n]+x_i[<n]\|+|\mu c +x_i[n]| \le 1 \,.
\end{equation}

\emph{Case 1: $x_i=(0,\dots,0,\pm 1)$. } Here (\ref{eq-part2}) amounts to $\mu+|\mu c \pm 1| \le 1$. That is $|\mu c \pm 1| \le 1-\mu$ and gives
\[
1-\mu \ge |\pm 1+\mu c| \ge ||\pm 1|-|\mu c||=|1-\mu c| \ge 1-\mu c\,.
\]
However, this yields $c \ge 1$, a contradiction.

\emph{Case 2: $x_i=(x_i[<n],0)$ with $x_i[<n] \in S^{n-2}$. } Now (\ref{eq-part2}) gives
\[
\|-\mu x_0[<n]+x_i[<n]\|+\mu c\le 1\,.
\]
Hence $\|-\mu x_0[<n]+x_i[<n]\|^2\le (1-\mu c)^2$. Expressing the norm by the inner product and using $\|x_0[<n]\|=\|x_i[<n]\|=1$, we obtain
\[
\mu^2-2\mu\langle x_0[<n],x_i[<n] \rangle +1 \le 1-2\mu c+\mu^2 c^2,\,
\]
and in turn
\[
\mu^2(1-c^2)+2\mu(c-\langle x_0[<n],x_i[<n] \rangle) \le 0\,.
\]
This is impossible, because $\mu > 0$, $0 \le c < 1$, and $c \ge \langle x_0[<n],x_i[<n] \rangle$.
\end{proof}


\section{Examples in \boldmath$\bR^2$\unboldmath}

($\alpha$) \emph{Let $K = B(o,1)$ be a circular disc. Then $K = B(c_K,R(K))$ and}
\[
i(K,\varrho) = \left\{\begin{array}{ll}
3,& 0 < \varrho \le R(K) = 1,\\
\infty,& \varrho > R(K) = 1\,.
\end{array}\right.
\]
In particular, $i(K,R(K)) < \infty$.

Figure~\ref{fig_alpha} illustrates the case $\varrho=1$: a point $x \in B(o,1)$ is $1$-$t$-illuminated by $r \in S^1$ if $x+1r \in B(o,1)$, i.e., if
$x \in B(o,1)-r$. $r_1$, $r_2$, $r_3$ $1$-$t$-illuminate the whole of $B(o,1)$.
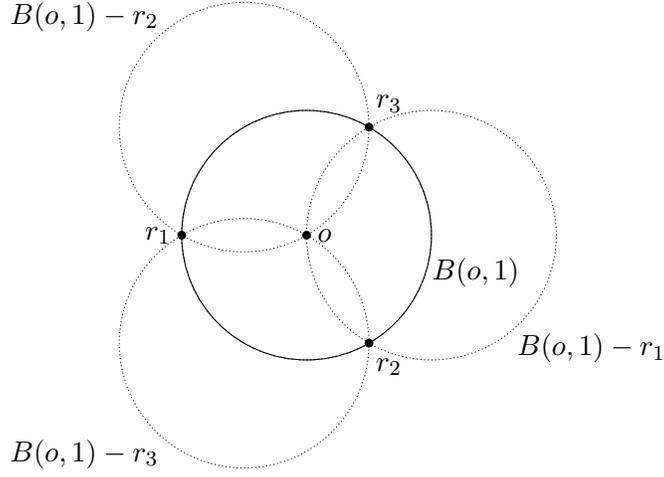
\begin{figure}
\begin{center}
\begin{tikzpicture}[scale=1.66]
  \draw (0,0) circle (1) (0,0) node[right] {\small $o$} (-1,0) node[left] {\small $r_1$} (.47,.9) node[above right] {\small $r_3$} (.47,-.9) node[below right] {\small $r_2$} (.92,-.3) node[right] {\small $B(o,1)$} (1.6,-.7) node[below right] {\small $B(o,1)-r_1$} (-1.1,-1.55) node[below left] {\small $B(o,1)-r_3$} (-1.1,1.55) node[above left] {\small $B(o,1)-r_2$};
  \fill (.5,.866) circle (.035) (.5,-.866) circle (.035) (0,0) circle (.035) (-1,0) circle (.035);
  \draw[densely dotted]  (-.5,.866) circle (1) (-.5,-.866) circle (1) (0,0) circle (1) (1,0) circle (1);
\end{tikzpicture}
\end{center}
\caption{$i(B(o,1),R(B^o(o,1)))=3 < \infty$}
\label{fig_alpha}
\end{figure}


($\beta$) \emph{If $T$ is a triangle, then we have} $i(T,R(T)) = \infty$.

\begin{proof} We use the second criterion from Corollary~\ref{cor4.2}. The set ${\rm vert}\,(T) \cap {\rm bd}(B(c_T,R(T)))$ consists of two vertices (if $T$ is obtuse) or  three vertices (otherwise) of $T$: $v_1$, $v_2$ (, $v_3$). The set $T + (c_T-v_i)$ covers at most an angular sector from $B^o(c_T,\delta)$, whose size $\gamma_i$ is that of the angle of $T$ at $v_i$ (see Figure~\ref{fig_beta}).
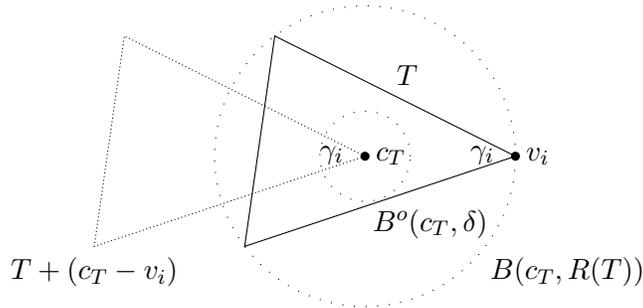
\begin{figure}
\begin{center}
\begin{tikzpicture}[scale=.4]
  \draw[loosely dotted] (0,0) circle (5) (0,0) circle (1.5);
  \draw (5,0) -- (-3,4) -- (-4,-3) -- cycle (0,0) node[right] {\small $c_T$} (5,0) node[right] {\small $v_i$} (-1.1,.05) node {\small $\gamma_i$} (3.9,.05) node {\small $\gamma_i$} (-.1,-1.4) node[below right] {\small $B^o(c_T,\delta)$} (3.8,-3) node[below right] {\small $B(c_T,R(T))$} (.7,2) node[above right] {\small $T$} (-9,-3) node[below] {\small $T+(c_T-v_i)$};
  \fill (0,0) circle (.15) (5,0) circle (.15);
  \draw[densely dotted] (0,0) -- (-8,4) -- (-9,-3) -- cycle;
\end{tikzpicture}
\end{center}
\caption{Proof of $i(T,R(T))= \infty$}
\label{fig_beta}
\end{figure}
Since the angles of $T$ sum up to $\pi$, the full circular disc $B^o(c_T,\delta)$ cannot be covered by $\bigcup \limits_{v \in {\rm vert}\,(T) \cap {\rm bd}\,(B(c_T,R(T)))} T + (c_T - v)$. Now Corollary~\ref{cor4.2} gives $i(T,R(T)) = \infty$.
\end{proof}

($\gamma$) \emph{A convex quadrangle $Q$ satisfies $i(Q,R(Q)) < \infty$ if and only if $Q$ is a rectangle. In that case we have $i(Q,R(Q)) = 4$.}

\begin{proof}
``$\Leftarrow$'': The illumination of a rectangle $Q$ is illustrated in Figure~\ref{fig_gamma1}. $c_Q$ is the midpoint, and $R(Q)$ is half the length of a diagonal.
\begin{figure}
\begin{center}
\begin{tikzpicture}[scale=.5]
  \draw (3,2) -- (-3,2) -- (-3,-2) -- (3,-2) -- cycle (0,0) node[below right] {\small $c_Q$} (3,2) node[right] {\small $v_1$} (3,2) node[right] {\small $v_1$} (-3,2) node[left] {\small $v_2$} (-3,-2) node[left] {\small $v_3$} (3,-2) node[right] {\small $v_4$} (1.5,-2) node[below] {\small $Q$} (6,2.5) node[right] {\small $Q+(c_Q-v_3)$} (-6,2.5) node[left] {\small $Q+(c_Q-v_4)$} (-6,-2.5) node[left] {\small $Q+(c_Q-v_1)$} (6,-2.5) node[right] {\small $Q+(c_Q-v_2)$};
  \fill (0,0) circle (.12) (3,2) circle (.12) (-3,2) circle (.12) (-3,-2) circle (.12) (3,-2) circle (.12);
  \draw[densely dotted] (0,0) -- (0,4) -- (-6,4) -- (-6,0) -- cycle (0,0) -- (0,-4) -- (6,-4) -- (6,0) -- cycle (0,4) -- (6,4) -- (6,0) (0,-4) -- (-6,-4) -- (-6,0);
\end{tikzpicture}
\end{center}
\caption{Illumination of a rectangle $Q$}
\label{fig_gamma1}
\end{figure}
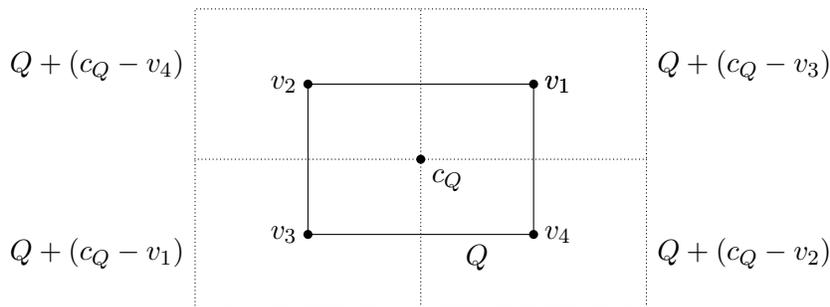

``$\Rightarrow$'': Let $v_1,v_2,v_3,v_4$ be the vertices of $Q$ and $\gamma_1,\gamma_2,\gamma_3,\gamma_4$ the corresponding angle measures.
By Corollary~\ref{cor4.2}, $i(Q,R(Q)) < \infty$ yields
\[
\exists \delta > 0\,\left( B^o (c_Q,\delta)
\subseteq \bigcup  \limits_{v \in \{v_1,v_2,v_3,v_4\}\cap {\rm bd}\,(B(c_Q,R(Q)))} Q + (c_Q-v) \right)\,.
\]

Like in $(\beta)$ one sees that $Q+(c_Q - v_i)$ covers a sector of $B^o(c_Q,\delta)$ of size $\gamma_i$. From this and
$\gamma_1+\gamma_2+\gamma_3+\gamma_4=2\pi$ it follows that one needs all four angles. Therefore $Q$ is a cyclic quadrangle with circumscribed circle
${\rm bd}\,(B(c_Q,R(Q)))$, and the translates of $Q$'s four angles to $c_Q$ have disjoint interiors and only four bounding rays. We obtain
\[
\gamma_1 + \gamma_2 = \gamma_2 + \gamma_3 = \gamma_3 + \gamma_4 = \gamma_4 + \gamma_1 =\pi
\]
(see Figure~\ref{fig_gamma2}).
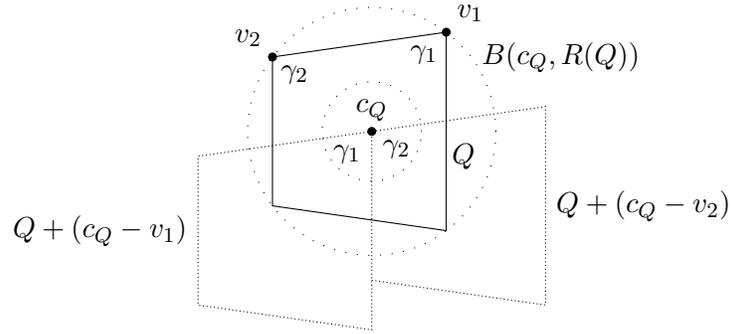
\begin{figure}
\begin{center}
\begin{tikzpicture}[scale=.33]
  \draw[loosely dotted] (0,0) circle (5) (0,0) circle (2);
  \draw (3,4) -- (-4,3) -- (-4,-3) -- (3,-4) -- cycle (0,0) node[above] {\small $c_Q$} (3,4) node[above right] {\small $v_1$} (-4,3) node[above left] {\small $v_2$} (0,-.1) node[below left] {\small $\gamma_1$} (0,.1) node[below right] {\small $\gamma_2$} (3.1,3.9) node[below left] {\small $\gamma_1$} (-4.1,3.1) node[below right] {\small $\gamma_2$} (2.8,-1) node[right] {\small $Q$} (4,3) node[right] {\small $B(c_Q,R(Q))$} (7,-3) node[right] {\small $Q+(c_Q-v_2)$} (-7,-4) node[left] {\small $Q+(c_Q-v_1)$};
  \fill (0,0) circle (.18) (3,4) circle (.18) (-4,3) circle (.18);
  \draw[densely dotted]  (0,0) -- (-7,-1) -- (-7,-7) -- (0,-8) -- cycle (0,-6) -- (7,-7) -- (7,1) -- (0,0);
\end{tikzpicture}
\end{center}
\caption{Proof of $\gamma_1+\gamma_2=\pi$}
\label{fig_gamma2}
\end{figure}
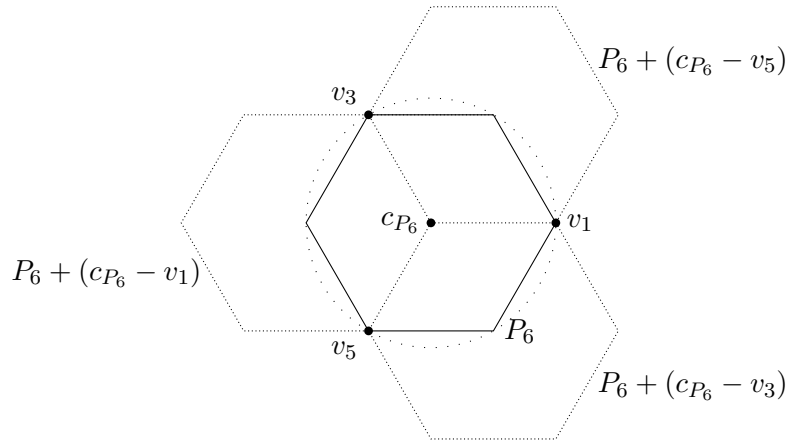
\begin{figure}
\begin{center}
\begin{tikzpicture}[scale=1.66]
  \draw[loosely dotted] (0,0) circle (1);
  \draw (1,0) -- (.5,.866) -- (-.5,.866) -- (-1,0) -- (-.5,-.866) -- (.5,-.866) -- cycle (0,0) node[left] {\small $c_{P_6}$} (1,0) node[right] {\small $v_1$} (-.5,.866) node[above left] {\small $v_3$} (-.5,-.866) node[below left] {\small $v_5$} (.5,-.866) node[right] {\small $P_6$} (1.25,1.3) node[right] {\small $P_6+(c_{P_6}-v_5)$} (1.25,-1.3) node[right] {\small $P_6+(c_{P_6}-v_3)$} (-1.75,-.4) node[left] {\small $P_6+(c_{P_6}-v_1)$};
  \fill (0,0) circle (.035) (1,0) circle (.035) (-.5,.866) circle (.035) (-.5,-.866) circle (.035);
  \draw[densely dotted]  (0,0) -- (1,0) -- (1.5,.866) -- (1,1.732) -- (0,1.732) -- (-.5,.866) -- cycle (0,0) -- (-.5,-.866) -- (-1.5,-.866) -- (-2,0) -- (-1.5,.866) -- (.5,.866) (1,0) -- (1.5,-.866) -- (1,-1.732) -- (0,-1.732) -- (-.5,-.866);
\end{tikzpicture}
\end{center}
\caption{$i(P_6,R(P_6))=3<\infty$}
\label{fig_delta}
\end{figure}
In addition, for the \emph{cyclic} quadrangle we have
\[
\gamma_1 + \gamma_3 = \gamma_2 + \gamma_4 = \pi\,.
\]
So, necessarily, $\gamma_1 = \gamma_2 = \gamma_3 = \gamma_4 = \frac{\pi}{2}$, and $Q$ is a rectangle.
\end{proof}

($\delta$) \emph{Every regular $n$-gon $P_n, n \ge 4$, satisfies} $i(P_n,R(P_n)) < \infty$.

This follows easily from Corollary~\ref{cor4.2}. For example, $i(P_6,R(P_6)) = 3$ (see Figure~\ref{fig_delta}).
%
%


\end{document}